\documentclass[11pt]{article}
\usepackage[height=210mm,width=145mm]{geometry}
\usepackage{amsmath}
\usepackage{amssymb}
\usepackage{amsthm}
\usepackage[usenames]{color}
\usepackage{amscd}
\usepackage{dsfont}
\usepackage{indentfirst}

\usepackage[colorlinks=true,linkcolor=blue,filecolor=red,
citecolor=webgreen]{hyperref}
\definecolor{webgreen}{rgb}{0,.5,0}

\def\C{{\mathds{C}}}

\def\N{{\mathds{N}}}
\def\Z{{\mathds{Z}}}

\def\1{{\bf 1}}

\def\lcm{\operatorname{lcm}}

\def\cale{{\cal E}}
\def\barr{\begin{array}}
\def\earr{\end{array}}
\def\dd{\displaystyle}

\numberwithin{equation}{section}

\newtheorem{thm}{Theorem}[section]
\newtheorem{lem}{Lemma}[section]
\newtheorem{exm}{Example}[section]
\newtheorem{cor}{Corollary}[section]
\newtheorem{prop}{Proposition}[section]
\newtheorem{rem}{Remark}[section]

\begin{document}

\title{\bf On the number of subgroups of a given exponent in a finite abelian group}
\author{Marius T\u arn\u auceanu and L\'aszl\'o T\'oth}
\date{}
\maketitle

\centerline{Publications de l'Institut Math\'ematique Beograd {\bf 101(115)} (2017), 121--133}

\begin{abstract} This paper deals with the number of subgroups of a given exponent in a finite abelian group. Explicit
formulas are obtained in the case of rank two and rank three abelian groups. An asymptotic formula is also presented.
\end{abstract}



\section{Introduction}

One of the most important problems of combinatorial abelian group theory is to determine
the number of subgroups of a finite abelian group. This topic has enjoyed a constant evolution
starting with the first half of the $20^{\mathrm{th}}$ century. Since a finite
abelian group is a direct product of abelian $p$-groups, the above
counting problem can be reduced to $p$-groups. Formulas which give the
number of subgroups of type $\mu$ of a finite $p$-group of type
$\lambda$ were established by S.~Delsarte \cite{Del1948}, P.~E.~Djubjuk \cite{Dju1948} and
Y.~Yeh \cite{Yeh1948}. An excellent survey on this subject together with connections to symmetric
functions was written by M.~L.~Butler \cite{But1994} in 1994.

Another way to find the total number of subgroups of finite abelian
$p$-groups was described by G.~Bhowmik \cite{Bho1996} by using
divisor functions of matrices. By invoking different arguments,
formulas in the case of rank two $p$-groups were obtained by G.~C\u
alug\u areanu \cite{Cal2004}, M.~T\u arn\u auceanu
\cite{Tar2007,Tar2010}, M.~Hampejs, N.~Holighaus, L.~T\'oth,
C.~Wiesmeyr \cite{HHTW2014}, L.~T\'oth \cite{Tot2014} and for rank
three $p$-groups by M.~Hampejs, L.~T\'oth \cite{HamTot2013},
J.-M.~Oh \cite{Oh2013}. Note that the papers
\cite{HHTW2014,HamTot2013,Tot2014} include also direct formulas for
the groups $\Z_m\times \Z_n$ and $\Z_m\times \Z_n\times \Z_r$,
respectively, where $m,n,r\in \N^*:= \{1,2,\ldots\}$ are arbitrary.

The purpose of the current paper is to count the number of subgroups
of a given exponent in a finite abelian $p$-group. Explicit formulas
are obtained for rank two and rank three $p$-groups. The numbers of
subgroups of exponent $p$, respectively $p^2$ in an arbitrary
$p$-group are also considered. We prove that if two finite abelian groups have the same number of
subgroups of any exponent, then they are isomorphic. We also deduce compact formulas for
the number of subgroups of a given exponent of the group $\Z_m\times
\Z_n$, where $m,n\in \N^*$ are arbitrary. Furthermore, we obtain an
exact formula for the sum of exponents of the subgroups of
$\Z_m\times \Z_n$, and an asymptotic formula for the arithmetic
means of exponents of the subgroups of $\Z_n\times \Z_n$.

For the proofs we use two different approaches. The first one is
based on the known formula for the number of subgroups of a given
type in an abelian $p$-group, given in terms of gaussian
coefficients (Theorem \ref{Th_subgroups_type}). The second method,
applicable only for rank two groups, uses the representation of the
subgroups of $\Z_m\times \Z_n$ obtained by the second author in
\cite{Tot2014} (Theorem \ref{Th_repr}).

Most of our notation is standard and will usually not be repeated here.
For basic notions and results on group theory we refer the reader to \cite{Suz}.

\section{First approach}

Let $G$ be a finite abelian group of order $n$ and $G=G_1\times G_2\times\cdots\times G_m$ be
the primary decomposition of $G$, where $G_i$ is a $p_i$-group ($i=1,2,\ldots,m$). For every divisor
$d=p_1^{\alpha_1}p_2^{\alpha_2}\cdots p_m^{\alpha_m}$ of $n$ we denote
\begin{equation*}
\cale_d(G)=\{H \leq G : \exp(H)=d\}.
\end{equation*}

Since the subgroups $H$ of $G$ are of type $H= H_1\times H_2\times\cdots\times H_m$ with $H_i\leq G_i$ ($i=1,2,\ldots,m$), we infer that
\begin{equation} \label{eq_prod}
|{\cale}_d(G)|= \prod_{i=1}^m |\cale_{p_i^{\alpha_i}}(G_i)|.
\end{equation}

Equality \eqref{eq_prod} shows that the problem of counting the number of
subgroups of exponent $d$ in $G$ is reduced to $p$-groups. So, in this section we will assume that $G$ is a finite abelian
$p$-group, that is a group of type $\mathbb{Z}_{p^{\lambda_1}} \times \mathbb{Z}_{p^{\lambda_2}}\times \cdots \times
\mathbb{Z}_{p^{\lambda_k}}$ with $\lambda_1\geq\lambda_2\geq \ldots \geq\lambda_k\ge 1$.
In this case we will say that $G$ is of type $\lambda$, where $\lambda$ is the partition $(\lambda_1,\lambda_2,\ldots,\lambda_k,0,\ldots)$,
and we will denote it by $G_{\lambda}$.

We recall the following well--known result, which gives the number of subgroups of type $\mu$ of $G_{\lambda}$ (see \cite{Del1948,Dju1948,Yeh1948}).

\begin{thm} \label{Th_subgroups_type} For every partition $\mu \preceq\lambda$ {\rm (}i.e., $\mu_i\leq\lambda_i$ for every $i\in \N^*${\rm )} the
number of subgroups of type $\mu$ in $G_{\lambda}$ is
\begin{equation*}
\alpha_{\lambda}(\mu;p) =\prod_{i\geq 1} p^{(a_i-b_i)b_{i+1}}\binom{a_i-b_{i+1}}{b_i-b_{i+1}}_{\hspace{-1mm}p}\,,
\end{equation*}
where $\lambda' =(a_1,a_2,\ldots,a_{\lambda_1},0,\ldots)$, $\mu'=(b_1,b_2,\ldots,b_{\mu_1},0,\ldots)$ are the partitions conjugate
to $\lambda$ and $\mu$, respectively, and
\begin{equation*}
\binom{n}{k}_{\hspace{-1mm}p} = \displaystyle\frac{\prod_{i=1}^n(p^i-1)}{\prod_{i=1}^k(p^i-1) \prod_{i=1}^{n-k}(p^i-1)}
\end{equation*}
is the gaussian binomial coefficient {\rm (}it is understood that $\prod_{i=1}^m (p^i-1)=1$ for $m=0${\rm )}.
\end{thm}

By using Theorem \ref{Th_subgroups_type} a way to compute the number of subgroups of exponent $p^i$ in $G_{\lambda}$ can be inferred, namely
\begin{equation*}
|\cale_{p^i}(G_{\lambda})| = \hspace{-3mm}\displaystyle\sum_{\mu\preceq\lambda,\, \mu_1=i}\hspace{-3mm}
\alpha_{\lambda}(\mu;p) \quad (i=0,1,\ldots,\lambda_1),
\end{equation*}
which is a polynomial in $p$ with integer coefficients. If $i\geq 1$ and $k\geq 2$ are arbitrary, then the polynomial
$|\cale_{p^i}(G_{\lambda})|$ can not be given explicitly, but we will do this in some
particular cases. Namely, we will consider the following cases: $k\in \{2,3\}$ and $i\geq 1$ arbitrary, $i\in \{1,2\}$ and $k\geq 2$ arbitrary.

We remark first that since $\sum_{H\in\cale_{p^i}(G_{\lambda})}H\leq G_{\lambda}$ and $\sum_{H\in\cale_{p^i}(G_{\lambda})}H$ is
of type $\mathbb{Z}_{p^i}\times \cdots\times \mathbb{Z}_{p^i}\times\mathbb{Z}_{p^{\lambda_r}}\times \cdots
\times\mathbb{Z}_{p^{\lambda_k}}$, where $r = \min\{j: \lambda_j <i\}$, we have
\begin{equation*}
|\cale_{p^i}(G_{\lambda})|= |\cale_{p^i}(\mathbb{Z}_{p^i}\times \cdots \times\mathbb{Z}_{p^i}
\times\mathbb{Z}_{p^{\lambda_r}}\times \cdots \times\mathbb{Z}_{p^{\lambda_k}})|
\end{equation*}
with the convention that the direct product $\mathbb{Z}_{p^{\lambda_r}}\times\cdots\times\mathbb{Z}_{p^{\lambda_k}}$ is trivial for
$i\leq\lambda_k$.

Our first result gives a precise expression for $|\cale_{p^i}(G_{\lambda})|$ in the case $k=2$.

\begin{prop} \label{Prop_exp_p} Let $G_{\lambda} =\mathbb{Z}_{p^{\lambda_1}}\times\mathbb{Z}_{p^{\lambda_2}}$ with $\lambda_1\geq\lambda_2\geq 1$. Then
\begin{align*}
|\cale_{p^i}(G_{\lambda})| = \begin{cases}
\dd \frac{p^{i+1}+p^i-2}{p-1}\,, &1\leq i\leq\lambda_2\,,\\ & \\
\dd \frac{p^{\lambda_2+1}-1}{p-1}\,, &\lambda_2<i\leq\lambda_1\,.
\end{cases}
\end{align*}
\end{prop}

\begin{proof} We have $\lambda=(\lambda_1,\lambda_2,0,\ldots)$ and consequently
\begin{align*}
\lambda'=(2,2,\ldots,2,1,1,\ldots,1,0,\ldots),
\end{align*}
where the number of $2$'s is $\lambda_2$ and the number of $1$'s is $\lambda_1-\lambda_2$. Then
$$
|\cale_{p^i}(G_{\lambda})|=\dd\sum_{j=0}^i N_{i,j},
$$
where
$N_{i,j}$ denotes the number of subgroups of type $(i,j)$ in $G_{\lambda}$, or equivalently in $\mathbb{Z}_{p^i}\times\mathbb{Z}_{p^i}$
if $i\leq\lambda_2$ or in $\mathbb{Z}_{p^i}\times \mathbb{Z}_{p^{\lambda_2}}$ if $\lambda_2<i\leq\lambda_1$.

In the first case we obtain $N_{i,j}=(p+1)p^{i-j-1}$ for $j=0,1,\ldots,i-1$, and $N_{i,i}=1$. Therefore
$$
|\cale_{p^i}(\mathbb{Z}_{p^i}\times\mathbb{Z}_{p^i})|=1+\dd\sum_{j=0}^{i-1}(p+1)p^{i-j-1} = \dd\frac{p^{i+1}+p^i-2}{p-1}\,.
$$
In the second case we obtain
$N_{i,j}=p^{\lambda_2-j}$ for $j=0,1,\ldots,\lambda_2$. Therefore
$$
|\cale_{p^i}(\mathbb{Z}_{p^i}\times\mathbb{Z}_{p^{\lambda_2}})|= \dd\sum_{j=0}^{\lambda_2}p^{\lambda_2-j}=\dd\frac{p^{\lambda_2+1}-1}{p-1}\,.
$$
This completes the proof.
\end{proof}

\begin{exm} We have
\begin{align*}
|\cale_{p^i}(\mathbb{Z}_{p^4}\times\mathbb{Z}_{p^2})|=\begin{cases}
1,&i=0,\\
p+2,&i=1,\\
p^2+2p+2,&i=2,\\
p^2+p+1,&i=3 \text{\rm{ or }} i=4. \end{cases}
\end{align*}
\end{exm}

In particular, by summing all quantities $|\cale_{p^i}(G_{\lambda})|$ ($i=0,1,\ldots,\lambda_1$) we obtain the total number of subgroups of
$G_{\lambda}$ (see also \cite[Prop.\ 2.9]{Tar2007}, \cite[Th.\ 3.3]{Tar2010}).

\begin{cor} \label{Cor_total_nr_subgroups_2} The total number of subgroups of $G_{\lambda}=\mathbb{Z}_{p^{\lambda_1}}\times\mathbb{Z}_{p^{\lambda_2}}$,
where $\lambda_1\geq\lambda_2\geq 1$, is
\begin{equation} \label{s_p}
\frac{1}{(p{-}1)^2}\left[(\lambda_1{-}\lambda_2{+}1)p^{\lambda_2{+}2}{-}(\lambda_1{-}\lambda_2{-}1) p^{\lambda_2{+}1}{-}(\lambda_1{+}\lambda_2{+}3)p{+}(\lambda_1{+}\lambda_2+1)\right].
\end{equation}
\end{cor}

\begin{exm} The total number of subgroups of $\mathbb{Z}_{p^4} \times \mathbb{Z}_{p^2}$ is
$3p^2+5p+7$.
\end{exm}

Now consider the case of rank three $p$-groups. We need the following lemma.

\begin{lem} \label{Lemma} Let $N_{i,j,\ell}$ denote the number of subgroups of type $(i,j,\ell)$ {\rm (}$i\geq j\geq \ell\geq 0${\rm )}
in $G_{\lambda} =\mathbb{Z}_{p^{\lambda_1}}\times \mathbb{Z}_{p^{\lambda_2}}\times \mathbb{Z}_{p^{\lambda_3}}$
with $\lambda_1\geq\lambda_2\geq \lambda_3\geq 1$. Then
\begin{align*}
N_{i,j,\ell}= \begin{cases}
p^{2i-2\ell-3}(p+1)(p^2+p+1), & \ell<j<i\leq \lambda_3,\\
p^{2(i-\ell-1)}(p^2+p+1), & \ell<j=i\leq \lambda_3 \text{ or } \ell=j<i\leq \lambda_3, \\
1, & \ell=j=i\leq \lambda_3, \\
p^{\lambda_3+i-2\ell-2}(p+1)^2, & \ell<j\leq \lambda_3 < i\leq \lambda_2, \\
p^{\lambda_3+i-2j-1}(p+1), & \ell=j\leq \lambda_3 < i\leq \lambda_2, \\
p^{2\lambda_3+i-j-2\ell-1}(p+1), & \ell \leq \lambda_3 < j<i\leq \lambda_2, \\
p^{2(\lambda_3-\ell)}, & \ell \leq \lambda_3 < j=i\leq \lambda_2, \\
p^{\lambda_2+2\lambda_3-j-2\ell}, & \ell \leq \lambda_3 < j\leq \lambda_2<i \leq \lambda_1, \\
p^{\lambda_2+\lambda_3-2\ell-1}(p+1), & \ell <j \leq \lambda_3 \leq \lambda_2 <i \leq \lambda_1, \\
p^{\lambda_2+\lambda_3-2\ell}, & \ell = j \leq \lambda_3 \leq \lambda_2 <i \leq \lambda_1.
\end{cases}
\end{align*}
\end{lem}

\begin{proof} We use Theorem \ref{Th_subgroups_type}. We distinguish the following three cases:

I. If $i\leq \lambda_3$, then $N_{i,j,\ell}$ is the number of subgroups of type $(i,j,\ell)$ in $\mathbb{Z}_{p^i}\times
\mathbb{Z}_{p^i} \times\mathbb{Z}_{p^i}$. Here we need to consider $\lambda=(i,i,i,0,\ldots)$ with $\lambda'=(3,3,\ldots,3,0\ldots)$, where
the number of $3$'s is $i$, and $\mu=(i,j,\ell,0,\ldots)$ with
$$\mu'=(3,3,\ldots,3,2,2,\ldots,2,1,1,\ldots,1,0,\ldots),$$ where the number of $3$'s is $\ell$, the number
of $2$'s is $j-\ell$, the number of $1$'s is $i-j$. In the subcase $\ell<j<i\leq \lambda_3$ we deduce
\begin{equation*}
N_{i,j,\ell}=(p^2)^{j-\ell-1}p(p+1)(p^2)^{i-j-1}(p^2+p+1)= p^{2i-2\ell-3}(p+1)(p^2+p+1).
\end{equation*}

The subcases $\ell<j=i\leq \lambda_3$, $\ell=j<i\leq \lambda_3$ and $\ell=j=i\leq \lambda_3$ are treated similar.

II. If $\lambda_3<i\leq \lambda_2$, then $N_{i,j,\ell}$ is the number of subgroups of type $(i,j,\ell)$ in $\mathbb{Z}_{p^i}\times
\mathbb{Z}_{p^i} \times\mathbb{Z}_{p^{\lambda_3}}$. We consider $\lambda=(i,i,\lambda_3,0,\ldots)$ with
$\lambda'=(3,3,\ldots,3,2,2,\ldots,2,0\ldots)$, where the number of $3$'s
is $\lambda_3$, the number of $2$'s is $i-\lambda_3$, and $\mu$, $\mu'$ like in the case I.

For example, in the subcase $\ell=j\leq \lambda_3<i\leq \lambda_2$
we obtain
\begin{equation*}
N_{i,j,\ell}=(p^2)^{\lambda_3-j}p^{i-\lambda_3-1}(p+1)= p^{\lambda_3+i-2j-1}(p+1).
\end{equation*}

III. If $\lambda_2<i\leq \lambda_1$, then $N_{i,j,\ell}$ is the number of subgroups of type $(i,j,\ell)$ in $\mathbb{Z}_{p^i}\times
\mathbb{Z}_{p^{\lambda_2}} \times\mathbb{Z}_{p^{\lambda_3}}$. We consider $\lambda=(i,\lambda_2,\lambda_3,0,\ldots)$
with $$\lambda'=(3,3,\ldots,3,2,2,\ldots,2,1,1,\ldots,1,0\ldots),$$
where the number of $3$'s is $\lambda_3$, the number of $2$'s is $\lambda_2-\lambda_3$, the number of $1$'s is $i-\lambda_2$,
and $\mu$, $\mu'$ like in the case I.
\end{proof}

\begin{prop} \label{Prop_exp_p_rank_3} Let $G_{\lambda} =\mathbb{Z}_{p^{\lambda_1}}\times\mathbb{Z}_{p^{\lambda_2}}\times \mathbb{Z}_{p^{\lambda_3}}$
with $\lambda_1\geq\lambda_2\geq \lambda_3\geq 1$. Then
\begin{align*}
|\cale_{p^i}(G_{\lambda})| = \begin{cases}
\dd \frac{p^{2i-1}\left((i+1)p^4+(i-1)p^3-p^2-(i+2)p-i\right)+3}{(p^2-1)(p-1)}\,, &1\leq i\leq\lambda_3,\\
&\\
\dd \frac{(\lambda_3+1)p^{\lambda_3+i}(p^2-1)(p+1)-2p^{2\lambda_3+2}+2}{(p^2-1)(p-1)}\,, &\lambda_3<i\leq\lambda_2,\\
&\\
\dd \frac{(\lambda_3+1)p^{\lambda_2+\lambda_3+1}(p^2-1)-p^{2\lambda_3+2}+1}{(p^2-1)(p-1)}\,, &\lambda_2<i\leq\lambda_1\,.
\end{cases}
\end{align*}
\end{prop}

\begin{proof} We have
\begin{equation*}
|\cale_{p^i}(G_{\lambda})|=\sum_{0\leq \ell \leq j\leq i} N_{i,j,\ell}
\end{equation*}
and use Lemma \ref{Lemma}.

Case I. If $i\leq \lambda_3$, then
\begin{align*}
|\cale_{p^i}(G_{\lambda})| = & N_{i,i,i} + \sum_{\ell =0}^{i-1} \sum_{j=\ell+1}^{i-1} N_{i,j,\ell} + \sum_{\ell=0}^{i-1} N_{i,i,\ell} +
\sum_{\ell=0}^{i-1} N_{i,\ell,\ell} \\
= & 1+ \sum_{\ell =0}^{i-1} \sum_{j=\ell+1}^{i-1} p^{2i-2\ell-3}(p+1)(p^2+p+1) + 2 \sum_{\ell=0}^{i-1} p^{2(i-\ell-1)}(p^2+p+1) \\
= & 1+ \frac{p^2+p+1}{(p^2-1)(p-1)}\left((i-1)p^{2i+1}-ip^{2i-1}+p \right) + 2(p^2+p+1) \frac{p^{2i}-1}{p^2-1} \\
= & \frac{p^{2i-1}((i+1)p^4+(i-1)p^3-p^2-(i+2)p-i)+3}{(p^2-1)(p-1)}\,,
\end{align*}
by direct computations.

II. If $\lambda_3<i\leq \lambda_2$, then we have
\begin{align*}
|\cale_{p^i}(G_{\lambda})| = & \sum_{\ell =0}^{\lambda_3} \sum_{j=\ell+1}^{\lambda_3} N_{i,j,\ell} + \sum_{\ell=0}^{\lambda_3}
\sum_{j=\lambda_3+1}^{i-1} N_{i,j,\ell} + \sum_{\ell=0}^{\lambda_3} N_{i,i,\ell} + \sum_{\ell=0}^{\lambda_3} N_{i,\ell,\ell} \\
= & \sum_{\ell =0}^{\lambda_3} \sum_{j=\ell+1}^{\lambda_3} p^{\lambda_3+i-2\ell-2} (p+1)^2  + \sum_{\ell=0}^{\lambda_3}
\sum_{j=\lambda_3+1}^{i-1} p^{2\lambda_3+i-j-2\ell-1}(p+1) \\
+ & \sum_{\ell=0}^{\lambda_3} p^{2(\lambda_3-\ell)} + \sum_{\ell=0}^{\lambda_3}  p^{\lambda_3+i-2\ell -1}(p+1) \\
= & \frac{\lambda_3 p^{i-2}(p+1)^2+p^{\lambda_3}+p^i+p^{i-1}}{p^{\lambda_3}(p^2-1)} \left(p^{2\lambda_3+2}-1 \right) \\
- & \frac{p^{i-2}\left(p^{2\lambda_3+2}-(\lambda_3+1)p^2+ \lambda_3
\right)}{p^{\lambda_3}(p-1)^2} +
\frac{(p^{2\lambda_3+2}-1)(p^{i-\lambda_3-1}-1)}{(p-1)^2},
\end{align*}
which gives the above formula.

III. Finally, if $\lambda_2<i\leq \lambda_1$, then
\begin{gather*}
 |\cale_{p^i}(G_{\lambda})| =  \sum_{\ell =0}^{\lambda_3} \sum_{j=\lambda_3+1}^{\lambda_2} N_{i,j,\ell} + \sum_{\ell=0}^{\lambda_3-1}
\sum_{j=\ell+1}^{\lambda_3} N_{i,j,\ell} + \sum_{\ell=0}^{\lambda_3} N_{i,\ell,\ell} \\
=  \sum_{\ell =0}^{\lambda_3} \sum_{j=\lambda_3+1}^{\lambda_2} p^{\lambda_2+2\lambda_3-j-2\ell} + \sum_{\ell=0}^{\lambda_3-1}
\sum_{j=\ell+1}^{\lambda_3} p^{\lambda_2+\lambda_3-2\ell-1} (p+1) + \sum_{\ell=0}^{\lambda_3} p^{\lambda_2+\lambda_3-2\ell} \\
 =  \frac{(p^{2\lambda_3+2}-1)(p^{\lambda_2-\lambda_3}-1)}{(p^2-1)(p-1)} + \frac{p^{\lambda_2-\lambda_3+1}}{(p^2-1)(p-1)}
\left(\lambda_3p^{2\lambda_3+2}-(\lambda_3+1)p^{2\lambda_3}+1\right)
\\  +  p^{\lambda_2-\lambda_3}\frac{p^{2\lambda_3+2}-1}{p^2-1},
\end{gather*}
leading to the given result.
\end{proof}

Note that Proposition \ref{Prop_exp_p_rank_3} is valid also in the case $\lambda_3=0$, when it reduces to Proposition \ref{Prop_exp_p}.

\begin{exm} We have
\begin{align*}
|\cale_{p^i}(\mathbb{Z}_{p^4}\times\mathbb{Z}_{p^2}\times\mathbb{Z}_{p^2})| = \begin{cases}
1,&i=0,\\
2p^2+2p+3,&i=1,\\
3p^4+4p^3+6p^2+3p+3,&i=2,\\
3p^4+2p^3+2p^2+p+1,&i=3 \text{\rm{ or }} i=4. \end{cases}
\end{align*}
\end{exm}

\begin{cor} \label{Cor_total_nr_subgroups_3} The total number of subgroups of $G_{\lambda}=\mathbb{Z}_{p^{\lambda_1}} \times
\mathbb{Z}_{p^{\lambda_2}} \times \mathbb{Z}_{p^{\lambda_3}}$, where $\lambda_1 \geq \lambda_2 \geq \lambda_3\geq 1$, is
\begin{equation*}
\frac{A}{(p^2-1)^2(p-1)},
\end{equation*}
where
\begin{align*}
A = & (\lambda_3+1)(\lambda_1-\lambda_2+1)p^{\lambda_2+\lambda_3+5} +
2(\lambda_3+1)p^{\lambda_2+\lambda_3+4}  \\
&  - 2(\lambda_3+1)(\lambda_1-\lambda_2)p^{\lambda_2+\lambda_3+3} -
2(\lambda_3+1)p^{\lambda_2+\lambda_3+2} \\
&+ (\lambda_3+1)(\lambda_1-\lambda_2-1)p^{\lambda_2+\lambda_3+1} -
(\lambda_1+\lambda_2-\lambda_3+3)p^{2\lambda_3+4} \\
& -2 p^{2\lambda_3+3} + (\lambda_1 + \lambda_2 - \lambda_3-1) p^{2\lambda_3+2}\\
& + (\lambda_1 +\lambda_2 + \lambda_3+5) p^2 + 2p -(\lambda_1 + \lambda_2 + \lambda_3 +1).
\end{align*}
\end{cor}

\begin{proof} The total number of subgroups of $G_{\lambda}$ is
\begin{equation*}
\sum_{i=0}^{\lambda_1} |\cale_{p^i}(G_{\lambda})|=  \sum_{0\leq i\leq \lambda_3} |\cale_{p^i}(G_{\lambda})| +
\sum_{\lambda_3<i\leq \lambda_2} |\cale_{p^i}(G_{\lambda})| + \sum_{\lambda_2 < i\leq \lambda_1} |\cale_{p^i}(G_{\lambda})|
\end{equation*}
and summing the quantities given in Proposition \ref{Prop_exp_p_rank_3} we deduce the result.
\end{proof}

We remark that an equivalent formula to that given in Corollary \ref{Cor_total_nr_subgroups_3} was obtained in \cite[Cor.\ 2.2]{Oh2013}
by using different arguments. Corollary \ref{Cor_total_nr_subgroups_3} is valid also in the case $\lambda_3=0$, when it reduces to Corollary
\ref{Cor_total_nr_subgroups_2}.

\begin{exm} The total number of subgroups of $\mathbb{Z}_{p^4} \times \mathbb{Z}_{p^2}\times \mathbb{Z}_{p^2}$ is
$9p^4+8p^3+12p^2+7p+9$.
\end{exm}

Next consider the number of subgroups of exponent $p$ (that is, the
number of elementary abelian subgroups) in $G_{\lambda}$, which
equals the total number of nontrivial subgroups of $(\Z_p)^k$. Since
$(\Z_p)^k$ is a $k$--dimensional linear space over $\Z_p$, the
number in question is exactly the total number of nonzero subspaces,
which is, as well--known, $\sum_{i=1}^k \binom{k}{i}_p$ (Galois
number). So we have the next result.

\begin{prop} \label{Prop_elem} Let $G_{\lambda}=\mathbb{Z}_{p^{\lambda_1}}
\times\mathbb{Z}_{p^{\lambda_2}}\times\cdots\times\mathbb{Z}_{p^{\lambda_k}}$
with $\lambda_1\geq\lambda_2\geq \ldots \geq\lambda_k\geq 1$. Then
\begin{equation*}
|\cale_p(G_{\lambda})|= \dd\sum_{r=1}^k \binom{k}{r}_{\hspace{-1mm}p}.
\end{equation*}
\end{prop}

We give a direct proof of this formula based on Theorem \ref{Th_subgroups_type}.

\begin{proof} We use Theorem \ref{Th_subgroups_type} in the case $\lambda=(1,1,\ldots,1,0,\ldots)$,
where the number of $1$'s is $k$, and $\mu=(1,1,\ldots,1,0,\ldots)$,
where the number of $1$'s is $r$ with $1\leq r\leq k$. Here
$\lambda'=(k,0,0,\ldots)$, $\mu'=(r,0,0,\ldots)$ and obtain that the
number of subgroups of type $\mu$ is $\binom{k}{r}_p$, while the
number of subgroups of exponent $1$ is exactly $\sum_{r=1}^k
\binom{k}{r}_p$.
\end{proof}

\begin{exm} We have for any $\lambda_1\geq\lambda_2\geq \lambda_3\geq \lambda_4\geq
1$,
\begin{equation*}
|\cale_p(\mathbb{Z}_{p^{\lambda_1}}\times \mathbb{Z}_{p^{\lambda_2}}
\times\mathbb{Z}_{p^{\lambda_3}} \times \mathbb{Z}_{p^{\lambda_4}})|
= p^4+3p^3+4p^2+ 3p+4.
\end{equation*}
\end{exm}

Concerning the number of subgroups of exponent $p^2$ in
$G_{\lambda}$ we have the next formula.

\begin{prop} \label{Prop_exp_p^2} Let $G_{\lambda}=\mathbb{Z}_{p^{\lambda_1}}\times\mathbb{Z}_{p^{\lambda_2}}
\times\cdots\times\mathbb{Z}_{p^{\lambda_k}}$ with
$\lambda_1\geq\lambda_2\geq \ldots \geq
\lambda_t>\lambda_{t+1}=\ldots = \lambda_k=1$, where $0\leq t\leq k$
is fixed {\rm (}$t=0$ if each $\lambda_j$ is $1$ and $t=k$ if each
$\lambda_j$ is $\geq 2${\rm )}. Then
\begin{equation*}
|\cale_{p^2}(G_{\lambda})|= \dd \sum_{\substack{1\leq r\leq t\\
0\leq s\leq k-r}} p^{r(k-r-s)} \binom{k-r}{s}_{\hspace{-1mm}p}
\binom{t}{r}_{\hspace{-1mm}p},
\end{equation*}
which is zero {\rm (}empty sum{\rm )} for $t=0$.
\end{prop}

\begin{proof} Let $t\geq 1$. Use Theorem \ref{Th_subgroups_type} for
$$
\lambda=(2,2,\ldots,2,1,1,\ldots,1,0,\ldots),
$$
where the number of
$2$'s is $t$ and the number of $1$'s is $k-t$ and
$$
\mu=(2,2,\ldots,2,1,1,\ldots,1,0,\ldots),
$$
where the number of
$2$'s is $r$ and the number of $1$'s is $s$ with $1\leq r\leq t$,
$0\leq s\leq k-r$. Now $\lambda'=(k,t,0,\ldots)$,
$\mu'=(r+s,r,0,\ldots)$ and obtain that the number of subgroups of
type $\mu$ is
\begin{equation*}
p^{r(k-r-s)} \binom{k-r}{s}_{\hspace{-1mm}p}
\binom{t}{r}_{\hspace{-1mm}p}
\end{equation*}
and the number of subgroups of exponent $p^2$ is deduced by summing
over $r$ and $s$.
\end{proof}

\begin{exm} {\rm ($k=4$, $t=2$)} We have
\begin{equation*}
|\cale_{p^2}(\mathbb{Z}_{p^4} \times \mathbb{Z}_{p^2} \times
\mathbb{Z}_{p} \times \mathbb{Z}_{p})| = p^5+5p^4+6p^3+4p^2+2p+2.
\end{equation*}
\end{exm}

In what follows let
$G_{\lambda}=\mathbb{Z}_{p^{\lambda_1}}\times\mathbb{Z}_{p^{\lambda_2}}
\times\cdots\times\mathbb{Z}_{p^{\lambda_k}}$ and $G_{\kappa}=
\mathbb{Z}_{p^{\kappa_1}}\times\mathbb{Z}_{p^{\kappa_2}}\times
\cdots\times\mathbb{Z}_{p^{\kappa_{\ell}}}$ be two finite abelian
$p$-groups, where $\lambda_1\geq\lambda_2\geq \ldots
\geq\lambda_k\geq 1$ and $\kappa_1\geq\kappa_2\geq \ldots
\geq\kappa_{\ell}\geq 1$. Assume that $G_{\lambda}$ and $G_{\kappa}$
have the same number of subgroups of exponent $p^i$ for every $i$,
i.e.
\begin{equation*}
|\cale_{p^i}(G_{\lambda})|=|\cale_{p^i}(G_{\kappa})| \quad (i\geq 0).
\end{equation*}

Then $\lambda_1=\kappa_1$. On the other hand, since the function
\begin{equation*}
f:\N^* \to \N^*, \quad f(k)=\sum_{i=1}^k \binom{k}{i}_p
\end{equation*}
is one--to--one, by Proposition \ref{Prop_elem} we infer that
$k=\ell$. Clearly, if $k=1$ one obtains $G_{\lambda}\cong
G_{\kappa}$. The same thing can be also said for $k=2$ and $k=3$ by
Propositions \ref{Prop_exp_p} and \ref{Prop_exp_p_rank_3}. Inspired
by these remarks, we state and prove the following result.

\begin{prop} \label{Prop_Isom} Two finite abelian $p$-groups $G_{\lambda}$ and $G_{\kappa}$ are isomorphic
if and only if they have the same number of subgroups of exponent
$p^i$, for every $i\geq 0$.
\end{prop}

\begin{proof} Let $\lambda = ( \lambda_1, \lambda_2,\ldots,\lambda_k,0,\ldots)$ and $\kappa=(\kappa_1,\kappa_2,\ldots,\kappa_{\ell},0,\ldots)$,
where $\lambda_k,\kappa_{\ell}>0$, be partitions such that for
$p$-groups $G_{\lambda}$ and $G_{\kappa}$ one has $|\cale_{p^i}(G_{\lambda})| = |\cale_{p^i}(G_{\kappa})|$, for every $i\geq 0$.
As noted before, $\lambda_1=\kappa_1$ and $k=\ell$ hold true.

Let us define $\lambda_{k+1} = \kappa_{k+1} = 0$. We prove, by reverse induction on $t\le k+ 1$  that
$\lambda_t = \kappa_t$. So, let us assume that $\lambda_i = \kappa_i$ for all $k + 1\geq i\geq t+1$, and prove that $\lambda_t = \kappa_t$.

Suppose to the contrary that w.l.o.g. $\lambda_t > \kappa_t$. Since $\lambda_1 = \kappa_1$, one has $t>1$. Consider the partitions
$\lambda^{\lambda_t}$ and $\kappa^{\lambda_t}$ defined with
\begin{align*}
\lambda^{\lambda_t} & = (\underbrace{\lambda_t,\ldots,\lambda_t}_t, \gamma_{t+1},\ldots,\gamma_k,0,\ldots), \\
\kappa^{\lambda_t} & = (\underbrace{\lambda_t,\ldots,\lambda_t}_s, \kappa_{s+1},\ldots,\kappa_t, \gamma_{t+1},\ldots,\gamma_k,0,\ldots),
\end{align*}
where $\gamma_i = \lambda_i = \kappa_i$, for $i\geq t+1$, and $s = \max \{j: \kappa_j \geq \lambda_t\}$. Note that by our assumption $s < t$, and
since $\kappa_1=\lambda_1\geq \lambda_t$ also $s\geq 1$.

From the definition of numbers $\alpha_{\omega}(\mu;p)$, it is clear that for any three partitions $\mu,\sigma,\tau$, which satisfy
$\mu \preceq\sigma \preceq \tau$, it holds
\begin{equation*}
\alpha_{\sigma}(\mu;p) \leq \alpha_{\tau}(\mu;p).
\end{equation*}

Using this remark, the fact that $\kappa^{\lambda_t} \precneqq \lambda^{\lambda_t}$ and $\alpha_{\lambda^{\lambda_t}}(\lambda^{\lambda_t};p)=1$,
one has
\begin{gather*}
1+ |\cale_{p^{\lambda_t}}(G_{\kappa})|= 1+ |\cale_{p^{\lambda_t}}(G_{\kappa^{\lambda_t}})|
 = 1+ \sum_{\mu \preceq \kappa^{\lambda_t},\ \mu_1=\lambda_t} \alpha_{\kappa^{\lambda_t}}(\mu;p) \\
 \leq \alpha_{\lambda^{\lambda_t}}(\lambda^{\lambda_t};p) + \sum_{\mu \preceq \kappa^{\lambda_t},\ \mu_1=\lambda_t}
\alpha_{\lambda^{\lambda_t}}(\mu;p)
 \leq \sum_{\mu \preceq \lambda^{\lambda_t},\ \mu_1=\lambda_t} \alpha_{\lambda^{\lambda_t}}(\mu;p)\\
 = |\cale_{p^{\lambda_t}}(G_{\lambda^{\lambda_t}})|= |\cale_{p^{\lambda_t}}(G_{\lambda})|,
\end{gather*}
a contradiction.
\end{proof}

\begin{cor} Two arbitrary finite abelian groups are isomorphic if and only if they have the same
number of subgroups of any exponent.
\end{cor}

Finally, we note that another interesting problem is to find the polynomial
$|\cale_{p^i}(G_{\lambda})|$ in the case $k=4$.

\section{Second approach}

We need the next result giving the representation of subgroups of
the group $\Z_m\times \Z_n$. For every $m,n\in \N^*$ let
\begin{gather*}
J_{m,n}:=\left\{(a,b,c,d,\ell)\in (\N^*)^5: a\mid m, b\mid a, c\mid
n, d\mid c, \frac{a}{b}=\frac{c}{d}, \right. \\ \left. \ell \le \frac{a}{b}, \,
\gcd\left(\ell,\frac{a}{b} \right)=1\right\}.
\end{gather*}

Note that here $\gcd(b,d)\cdot \lcm(a,c)=ad$ and $\gcd(b,d)\mid \lcm(a,c)$.

For $(a,b,c,d,\ell)\in J_{m,n}$ define
\begin{equation*}
K_{a,b,c,d,\ell}:= \left\{\left(i\frac{m}{a}, i\ell \frac{n}{c}+j\frac{n}{d}\right): 0\le i\le a-1, 0\le
j\le d-1\right\}.
\end{equation*}

\begin{thm} \label{Th_repr} {\rm (\cite[Th.\ 3.1]{Tot2014})}  Let $m,n\in \N^*$.

i) The map $(a,b,c,d,\ell)\mapsto K_{a,b,c,d,\ell}$ is a bijection
between the set $J_{m,n}$ and the set of subgroups of $(\Z_m \times
\Z_n,+)$.

ii) The invariant factor decomposition of the subgroup
$K_{a,b,c,d,\ell}$ is
\begin{equation*}
K_{a,b,c,d,\ell} \simeq \Z_{\gcd(b,d)} \times \Z_{\lcm(a,c)}.
\end{equation*}

iii) The order of the subgroup $K_{a,b,c,d,\ell}$ is $ad$ and its exponent is $\lcm(a,c)$.
\end{thm}

Let $s_E(m,n)$ stand for the number of subgroups of exponent $E$ of the group $\Z_m\times \Z_n$.

\begin{prop} \label{Prop_m_n} For every $m,n\in \N^*$, $E\mid \lcm(m,n)$ we have
\begin{align}
s_E(m,n) & = \sum_{\substack{i\mid m, j\mid n\\ \lcm(i,j)=E}} \gcd(i,j) \label{1_1} \\
         & = \frac1{E} \sum_{\substack{i\mid m, j\mid n\\ \lcm(i,j)=E}} ij \label{1_2}.
\end{align}
\end{prop}

\begin{proof} According to Theorem \ref{Th_repr},
\begin{align*}
s_E(m,n)= \sum_{\substack{a\mid m\\ b\mid a}} \sum_{\substack{c\mid n\\ d\mid c}} \sum_{\substack{a/b=c/d=e \\ \lcm(a,c)=E}} \phi(e),
\end{align*}
where $\phi$ is Euler's totient function. This can be written (with $m=ax$, $a=by$, $n=cz$, $c=dt$) as

\begin{gather*}
s_E(m,n) = \sum_{\substack{bxe=m \\ dze=n \\ e \lcm(b,d)=E}} \phi(e)
 = \sum_{\substack{ix=m \\ jz=n}} \sum_{\substack{be=i\\ de=j\\ e \lcm(b,d)=E}} \phi(e) \\
 = \sum_{\substack{i\mid m \\ j\mid n \\ \lcm(i,j)=E}} \sum_{e\mid \gcd(i,j)}  \phi(e)
 = \sum_{\substack{i\mid m \\ j\mid n \\ \lcm(i,j)=E}} \gcd(i,j),
\end{gather*}
which is \eqref{1_1}. Formula \eqref{1_2} is its immediate consequence.
\end{proof}

\begin{rem} {\rm Proposition \ref{Prop_exp_p} is a direct consequence of the above result. The total number $s(m,n)$ of subgroups of the group $\Z_m\times \Z_n$ is (see \cite[Th.\ 3]{HHTW2014}, \cite[Th.\ 4.1]{Tot2014})
\begin{align} \label{s_m_n}
s(m,n) = \sum_{i\mid m, j\mid n} \gcd(i,j)
\end{align}
and \eqref{1_1} shows the distribution of the number of subgroups according to their exponents. Formula \eqref{s_p} can be obtained also
by using \eqref{s_m_n}.}
\end{rem}

\begin{exm} The total number of subgroups of $\Z_{12}\times \Z_{18}$ is $s(12,18)=80$ and we have $s_1(12,18)=1$, $s_2(12,18)=4$,
$s_3(12,18)=5$, $s_4(12,18)=3$, $s_6(12,18)=20$, $s_9(12,18)=4$, $s_{12}(12,18)=15$, $s_{18}(12,18)=16$, $s_{36}(12,18)=12$.
\end{exm}

\begin{cor} \label{Th_n_n} {\rm ($m=n$)} For every $n\in \N^*$ and $E\mid n$,
\begin{align} \label{1_3}
s_E(n,n) & = \sum_{\substack{i\mid E, j\mid E \\ \gcd(E/i,E/j)=1}} \gcd(i,j),
\end{align}
which equals the number of cyclic subgroups of the group $\Z_E\times \Z_E$.
\end{cor}

The fact that $s_E(n,n)$ equals the number of cyclic subgroups of the group $\Z_E\times \Z_E$, but without
deriving formula \eqref{1_3} is \cite[Th.\ 8]{HHTW2014}, proved by different arguments.

\begin{proof} In the case $m=n$, for every $E\mid n$ we have by \eqref{1_1},
\begin{gather*}
s_E(n,n)  = \sum_{\substack{i\mid n, j\mid n \\ \lcm(i,j)=E}} \gcd(i,j)
= \sum_{\substack{ia=E, jb=E \\ \lcm(E/a,E/b)=E}} \gcd(i,j) \\
= \sum_{\substack{ia=E, jb=E \\ \gcd(a,b)=1}} \gcd(i,j),
\end{gather*}
giving \eqref{1_3}, which equals the number of cyclic subgroups of the group $\Z_E\times \Z_E$ by
\cite[eq.\ (16)]{HHTW2014}.
\end{proof}

\begin{prop} \label{Prop_sum_exp} For every $m,n\in \N^*$ the sum of exponents of the subgroups of $\Z_m\times \Z_n$
is $\sigma(m)\sigma(n)$, where $\sigma(k)=\sum_{d\mid k} d$.
\end{prop}

\begin{proof} By \eqref{1_2} the sum of exponents of the subgroups of $\Z_m\times \Z_n$ is
\begin{gather*}
\sum_{E\mid \lcm(m,n)} E s_E(m,n)  = \sum_{E\mid \lcm(m,n)} E \cdot \frac1{E}  \sum_{\substack{i\mid m, j\mid n\\ \lcm(i,j)=E}} ij
 = \sum_{\substack{i\mid m, j\mid n\\ \lcm(i,j)\mid \lcm(m,n)}} ij \\
 = \sum_{i\mid m} i \sum_{j\mid n} j  = \sigma(m)\sigma(n).
\end{gather*}
\end{proof}

By Proposition \ref{Prop_sum_exp} the arithmetic mean of exponents of the subgroups of $\Z_m\times \Z_n$ is
$\sigma(m)\sigma(n)/s(m,n)$, where $s(m,n)$ is given by \eqref{s_m_n}. Now consider the case $m=n$.
Let $A(n)$ stand for the arithmetic mean of exponents of the subgroups of $\Z_n\times \Z_n$. We have
\begin{align} \label{A}
A(n)= \frac{\sigma(n)^2}{s(n)},
\end{align}
where $s(n):=s(n,n)$.

Recall that a function $f:\N^* \to \C$ is said to be multiplicative if $f(nn')=f(n)f(n')$ whenever $\gcd(n,n')=1$. It
is well known that the sum-of-divisors function $\sigma$ is multiplicative. The function $s(n)=\sum_{i,j\mid n} \gcd(i,j)$ is
also multiplicative, as shown by the following direct proof: Let $\gcd(n,n')=1$. Then
\begin{gather*}
s(nn')  = \sum_{i,j\mid nn'} \gcd(i,j) =  \sum_{\substack{a,b\mid n\\  a',b'\mid n'}} \gcd(aa',bb') \\
 = \sum_{a,b\mid n} \gcd(a,b) \sum_{a',b'\mid n'} \gcd(a',b') =s(n)s(n').
\end{gather*}

We conclude that the function $A$ given by \eqref{A} is
multiplicative and and for every prime power $p^{\nu}$,
\begin{align*}
A(p^{\nu})= \frac{(p^{\nu+1}-1)^2}{p^{\nu+2}+p^{\nu+1}-(2\nu+3)p+2\nu+1},
\end{align*}
cf. Corollary \ref{Cor_total_nr_subgroups_2}.

Since the exponent of every subgroup of $\Z_n\times \Z_n$ is a
divisor of $n$, whence $\leq n$, we deduce that $A(n)\leq n$ ($n\in
\N^*$). For the function $n\mapsto f(n):=A(n)/n \in (0,1]$ the
series taken over the primes
\begin{equation*}
\sum_p \frac{1-f(p)}{p}= \sum_p \frac{p-1}{p^2(p+3)}
\end{equation*}
is convergent, and it follows from a theorem of H.~Delange (see, e.g., \cite{Pos1988})
that the function $f$ has a non-zero mean value $M$ given by
\begin{align*}
M:= & \lim_{x\to \infty} \frac1{x} \sum_{n\le x} f(n) \\
= & \prod_p \left(1-\frac1{p}\right) \sum_{\nu=0}^{\infty} \frac{f(p^\nu)}{p^\nu}\\
= & \prod_p \left(1-\frac1{p}\right) \sum_{\nu=0}^{\infty} \frac{(p^{\nu+1}-1)^2}{p^{2\nu}(p^{\nu+2}+p^{\nu+1}-(2\nu+3)p+2\nu+1)}.
\end{align*}
the products being over the primes.

We prove the following more exact result.

\begin{prop} We have
\begin{align} \label{asympt}
\sum_{n\le x} A(n) = \frac{M}{2} x^2 + O\left(x \log^3  x \right).
\end{align}
\end{prop}

\begin{proof} Let $f(n)=\sum_{d\mid n} g(d)$ ($n\in \N^*$), that is $g=\mu*f$ in terms of the Dirichlet convolution, where $\mu$ is the
M\"{o}bius function. Here $g(p^\nu)=f(p^\nu)-f(p^{\nu-1})$ for every prime power $p^{\nu}$ ($\nu \in
\N^*$). Note that
\begin{equation*}
A(p^k)= \frac{(\sum_{i=0}^k p^i)^2}{\sum_{i=0}^k (2i+1)p^{k-i}}
\quad (k\geq 0),
\end{equation*}
so if we denote $S=\sum_{i=0}^{\nu-1} p^i$ and $T=\sum_{i=0}^{\nu-1}
(2i+1) p^{\nu-1-i}$, we have
\begin{equation*}
g(p^{\nu})= \frac{(Sp+1)^2}{p^{\nu}(Tp+2\nu +1)}-
\frac{S^2}{p^{\nu-1}T}=
\frac{2STp+T-pS^2(2\nu+1)}{p^{\nu}T(Tp+2\nu+1)}.
\end{equation*}

Since $S\leq T$ and $T\leq (2\nu-1)S$, we have
\begin{align*}
|g(p^{\nu})| & < \frac{\max \{2STp+T-pS^2(2\nu-1), pS^2(2\nu+1)-2STp\}}{p^{\nu}T(Tp+2\nu+1)} \\
& \leq \frac1{p^{\nu}} \max \left\{\frac{STp+T}{T(Tp+1)},\frac{pS^2(2\nu-1)}{T^2p} \right\} \\
& \leq \frac1{p^{\nu}} \max \{1,2\nu-1 \} = \frac{2\nu-1}{p^{\nu}},
\end{align*}
valid for every prime power $p^\nu$ ($\nu \in
\N^*$). Hence, $|g(n)|\le \tau(n^2)/n$ for every $n\geq 1$, where
$\tau(k)$ stands for the number of positive divisors of $k$.

We deduce that
\begin{equation*}
\sum_{n\le x} f(n)= \sum_{de\le x} g(d)= \sum_{d\le x} g(d) \sum_{e\le x/d} 1= \sum_{d\le x} g(d) \left( x/d +O(1)\right)
\end{equation*}
\begin{equation*}
=x \sum_{d=1}^{\infty} \frac{g(d)}{d} + O\left(x\sum_{d>x} \frac{|g(d)|}{d}\right) + O \left(\sum_{d\le x} |g(d)|\right)
\end{equation*}
\begin{equation*}
=M x  + O\left(x\sum_{d>x} \frac{\tau(d^2)}{d^2}\right) + O\left(\sum_{d\le x} \frac{\tau(d^2)}{d}\right),
\end{equation*}
where in the main term the coefficient of $x$ is $M$ by Euler's product formula. It is known that $\sum_{n\le x}
\tau(n^2)= cx\log^2 x+O(x\log x)$ with a certain constant $c$ and partial summation shows that
$\sum_{n>x} \tau(n^2)/n^2= O((\log^2 x)/x)$, $\sum_{n\le x} \tau(n^2)/n= O(\log^3 x)$.
Therefore,
\begin{align} \label{f}
\sum_{n\le x} f(n) = M x + O \left(\log^3  x \right).
\end{align}

Now \eqref{asympt} follows from \eqref{f} by partial summation.
\end{proof}

Formula \eqref{asympt} and its proof are similar to those of \cite[Th.\ 3.1.3]{TT2015}.

\medskip

{\bf Acknowledgement:} The authors thank the referee for very careful reading of the manuscript, many useful comments and for the proof of
Proposition \ref{Prop_Isom}.

\vskip3mm

\noindent Marius T\u arn\u auceanu \\
Faculty of  Mathematics \\
``Al.I. Cuza'' University \\
Ia\c si, Romania \\
e-mail: {\tt tarnauc@uaic.ro}

\vskip3mm

\noindent L\'aszl\'o T\'oth \\
Department of Mathematics\\
University of P\'{e}cs\\
P\'ecs, Hungary \\
e-mail: {\tt ltoth@gamma.ttk.pte.hu}

\end{document}